\documentclass[12pt,oneside,reqno]{amsart} 
\usepackage{etex}
\usepackage[T2A]{fontenc}
\usepackage[cp1251]{inputenc}
\usepackage{pifont}
\usepackage[dvips]{graphicx}
\usepackage{amsmath,latexsym,amsthm,cmap,indentfirst,setspace,ccaption,amssymb,color,amscd,mathrsfs,hyperref}
\usepackage{array,enumitem}

\usepackage{fourier}

\usepackage{amssymb,lipsum}
\usepackage{float,xypic}
\usepackage{cmap,longtable}

\usepackage{geometry}
\usepackage{caption, subcaption}

\usepackage{verbatim}
\hypersetup{
	bookmarks=true,         
	unicode=true,          
	pdftoolbar=true,        
	pdfmenubar=true,        
	pdffitwindow=false,     
	pdfnewwindow=true,      
	colorlinks=true,       
	linkcolor=blue,          
	citecolor=blue,        
	filecolor=magenta,      
	urlcolor=cyan           
}

\DeclareMathOperator{\Cr}{Cr}

\DeclareMathOperator{\Spec}{Spec}

\DeclareMathOperator{\Aut}{Aut}

\DeclareMathOperator{\J}{J}

\DeclareMathOperator{\PGL}{PGL}
\DeclareMathOperator{\GL}{GL}
\DeclareMathOperator{\SL}{SL}
\DeclareMathOperator{\PSL}{PSL}
\DeclareMathOperator{\OO}{O}
\DeclareMathOperator{\Cyc}{C}
\DeclareMathOperator{\PO}{PO}

\DeclareMathOperator{\Bir}{Bir}

\DeclareMathOperator{\Heis}{\mathcal H}
\DeclareMathOperator{\Quad}{Q}

\theoremstyle{plain}
\newtheorem{thm}{Theorem}[section]
\newtheorem{lem}[thm]{Lemma} 
\newtheorem{sled}[thm]{Corollary} 
\newtheorem{prop}[thm]{Proposition}

\theoremstyle{definition}
\newtheorem{mydef}[thm]{Definition} 
\newtheorem{rem}[thm]{Remark}

\makeatletter
\g@addto@macro{\endabstract}{\@setabstract}
\newcommand{\authorfootnotes}{\renewcommand\thefootnote{\@fnsymbol\c@footnote}}%
\makeatother

\begin{document}
	
	\begin{center}
		\LARGE 
		The Jordan Constant For Cremona Group of Rank 2 \par \bigskip
		
		\normalsize
		\authorfootnotes
		\large Egor Yasinsky\footnote{yasinskyegor@gmail.com} \par \bigskip
		\normalsize
		Steklov Mathematical Institute of Russian Academy of Sciences \\8 Gubkina st., Moscow, Russia, 119991\par \bigskip
		
	\end{center}

\newcommand{\CC}{\mathbb C} 
\newcommand{\FF}{\mathbb F} 
\newcommand{\RR}{\mathbb R}
\newcommand{\QQ}{\mathbb Q}
\newcommand{\ZZ}{\mathbb Z}
\newcommand{\PP}{\mathbb P}
\newcommand{\kk}{\mathbb k}
\newcommand{\Sph}{\mathbb S}
\newcommand{\Torus}{\mathbb T}
\newcommand{\Sym}{\mathfrak S}
\newcommand{\Alt}{\mathfrak A}
\newcommand{\Dih}{\mathrm D}
\newcommand{\XC}{X_{\mathbb C}}
\newcommand{\Weyl}{\mathscr W}
\newcommand{\id}{\mathrm{id}}
\newcommand{\Inv}{\Phi}

	\newcommand*\conj[1]{\overline{#1}}
	
	\begin{abstract}
		We compute the Jordan constant for the group of birational automorphisms of a projective plane $\mathbb{P}^2_{\kk}$, where $\kk$ is either an algebraically closed field of characteristic 0, or the field of real numbers, or the field of rational numbers.
	\end{abstract}
	
\section{Introduction}

\subsection{Jordan property}

Unless stated otherwise, $\kk$ denotes an algebraically closed field of characteristic zero. We start with the main definition of this article, first introduced by V. L. Popov in \cite[Definition 2.1]{Pop-Makar} (see also \cite[Definition 1]{Pop-Diff}).

\begin{mydef}
	A group $\Gamma$ is called {\it Jordan} (we also say that $\Gamma$ has {\it Jordan property}) if there exists a positive integer $m$ such that every finite subgroup $G\subset \Gamma$ contains a normal abelian subgroup $A\triangleleft G$ of index at most $m$. The minimal such $m$ is called the {\it Jordan constant} of $\Gamma$ and is denoted by $\J(\Gamma)$.
\end{mydef}

Informally, this means that all finite subgroups of $\Gamma$ are ``almost'' abelian. The name of $\J(\Gamma)$ (and the corresponding property) is justified by the classical theorem of Camille Jordan (for a modern exposition see e.g. \cite[Theorem 36.13]{CR}).

\begin{thm}[C. Jordan, 1878]\label{thm: Jordan thm}
	The group $\GL_n(\kk)$ is Jordan for every $n$.
\end{thm}

Since any subgroup of a Jordan group is obviously Jordan, Theorem \ref{thm: Jordan thm} implies that every linear algebraic group over $\kk$ is Jordan. In recent years the Jordan property has been studied for groups of birational automorphisms of algebraic varieties. The first significant result in this direction belongs to J.-P. Serre. Before stating it, let us recall that the {\it Cremona group} $\Cr_n(\kk)$ of rank $n$ is the group of birational automorphisms of a projective space $\mathbb{P}^n_{\kk}$ (or, equivalently, the group of $\kk$-automorphisms of the field $\kk(x_1,\ldots, x_n)$ of rational functions in $n$ independent variables). Note that $\Cr_1(\kk)\cong\PGL_2(\kk)$ is linear and hence is Jordan, but the group $\Cr_2(\kk)$ is already ``very far'' from being linear. However, the following holds.

\begin{thm}[J.-P. Serre {\cite[Theorem 5.3]{serre1},\cite[Th\'{e}or\`{e}me 3.1]{serre2}}]\label{thm: Cr2 is Jordan}
	The Cremona group $\Cr_2(\kk)$ over a field $\kk$ of characteristic 0 is Jordan.
\end{thm}

A far-going generalization of Theorem \ref{thm: Cr2 is Jordan} was recently proved by Yu. Prokhorov and C. Shramov. 

\begin{thm}\cite[Theorem 1.8]{PS-BAB}\label{thm: PS-BAB}
	For every positive integer $n$, there exists a constant $I=I(n)$ such that for any rationally connected variety $X$ of dimension $n$ defined	over an arbitrary field $\kk$ of characteristic 0 and for any finite subgroup $G\subset \Bir(X)$ there exists a normal abelian subgroup $A\subset G$ of index at most $I$.
\end{thm}

Note that this theorem states not only that $\Bir(X)$ are Jordan, but also that the corresponding constant may be chosen uniformly for all rationally connected $X$ of a fixed dimension.

\begin{sled}
	The group $\Cr_n(\kk)$ is Jordan for each $n\geqslant 1$.
\end{sled}

\begin{rem}
	Theorem \ref{thm: PS-BAB} was initially proved modulo so-called Borisov-Alexeev-Borisov conjecture, which states that for a given positive integer $n$, Fano varieties of dimension $n$ with terminal singularities are bounded, i. e. are contained in a finite number of algebraic families. This conjecture was settled in dimensions $\leqslant 3$ a long time ago, but in its full generality was proved only recently in a preprint of Caucher Birkar \cite{BAB}.
\end{rem}

\subsection{Jordan constant} So far we discussed only the Jordan property itself. Of course, after establishing that a given group is Jordan, the next natural question is to estimate its Jordan constant. This can be highly non-trivial: the precise values of $\J(\GL_n(\kk))$ for all $n$ were found only in 2007 by M. J. Collins \cite{J for Gl}. As for Cremona groups, in \cite{serre1} a {\it ``multiplicative''} upper bound for $\J(\Cr_2(\kk))$ is given. Specifically, it was shown that every finite subgroup $G\subset\Cr_2(\kk)$ contains a normal abelian subgroup $A$ with $[G:A]$ dividing $2^{10}\cdot 3^4\cdot 5^2\cdot 7$. 

For any group $\Gamma$ one can consider a closely related constant $\overline{\J}(\Gamma)$, which is called a {\it weak Jordan constant} in \cite{PS-J3}. By definition, it is the minimal number $m$ such that for any finite subgroup $G\subset \Gamma$ there exists a {\it not necessarily normal} abelian subgroup $A\subset G$ of index at most $m$. The following beautiful fact highlights the connection between $\J(\Gamma)$ and $\overline{\J}(\Gamma)$ (see e.g. \cite[Theorem 1.41]{Isaacs}).

\begin{thm}[Chermak-Delgado]\label{rem: Chermak-Delgado}
	If $G$ is a finite group and $A$ is its abelian subgroup, then one can find a characteristic abelian subgroup $N$ of G such that
	\[
	[G:N]\leqslant [G:A]^2.
	\]
\end{thm}
	
It follows from Chermak-Delgado theorem that
	\[
	\overline{\J}(\Gamma)\leqslant\J(\Gamma)\leqslant\overline{\J}(\Gamma)^2
	\]
for any Jordan group $\Gamma$. The weak Jordan constants for $\Cr_2(\kk)$ and $\Cr_3(\kk)$ were computed by Prokhorov and Shramov.

\begin{thm}\cite[Proposition 1.2.3, Theorem 1.2.4]{PS-J3}
	Suppose that the field $\kk$ has characteristic 0. Then one has
	\[
	\overline{\J}(\Cr_2(\kk))\leqslant 288,\ \ \ \ \ \ \overline{\J}(\Cr_3(\kk))\leqslant 10\ 368.
	\]
	These bounds become equalities if the base field $\kk$ is algebraically closed.
\end{thm}

As we noticed above, one can also obtain an upper bound for $\J(\Cr_2(\kk))$ from this theorem, namely \[\J(\Cr_2(\kk))\leqslant 82\ 944=288^2.\] We will show that this bound is very far from being sharp. The goal of this paper is to compute an exact Jordan constant for the plane Cremona group. Our first main theorem is

\begin{thm}\label{thm: main}
	Let $\kk$ be an algebraically closed field of characteristic 0. Then
	\[
	\J(\Cr_2(\kk))=7200.
	\]
\end{thm}

Further, we compute the Jordan constant for Cremona group $\Cr_2(\RR)$ and, as a by-product, for the group $\Cr_2(\QQ)$.

\begin{thm}\label{thm: J for Real Cremona}
	One has
	\[
	\J(\Cr_2(\RR))=120,\ \ \ \ \overline{\J}(\Cr_2(\RR))=20.
	\]
\end{thm} 

\begin{thm}\label{thm: J for Q Cremona}
	One has
	\[
	\J(\Cr_2(\QQ))=120,\ \ \ \ \overline{\J}(\Cr_2(\QQ))=20.
	\]
\end{thm}

Let us consider the category whose objects are $\RR$-schemes with non-empty real loci and morphisms are defined as follows: we say that there is a morphism $f:X\dashrightarrow Y$ if $f$ is a rational map defined at all {\it real} points of $X$. Automorphisms in such a category are called {\it birational diffeomorphisms} and the corresponding group is denoted by $\Aut(X(\RR))$. In recent years, birational diffeomorphisms of real rational projective surfaces have been studied intensively. As a by-product of Theorem \ref{thm: J for Real Cremona}, we will get the Jordan constant for the group of birational diffeomorphisms of $\PP_\RR^2$ and the sphere $\Sph^2$ viewed as the real locus $\Quad_{3,1}(\RR)$ of the 2-dimensional quadric
\[
\Quad_{3,1}=\big\{[x_0:x_1:x_2:x_3]: x_0^2+x_1^2+x_2^2-x_3^2=0\big \}\subset\PP_\RR^{3}.
\] 
\begin{thm}\label{thm: Bir-Diff}
	The following holds:
	\begin{align*}
	\J\big(\Aut(\PP^2(\RR))\big)=60,&\ \ \ \ \ \ \overline{\J}\big(\Aut(\PP^2(\RR))\big)=12,  \\
	\J(\Quad_{3,1}(\RR))=60,&\ \ \ \ \ \ \overline{\J}(\Quad_{3,1}(\RR))=12.
	\end{align*}
\end{thm}

\subsection*{Notation} Our notation is mostly standard.

\begin{itemize}
	\item $\Sym_n$ denotes the symmetic group of degree $n$;
	\item $\Alt_n$ denotes the alternating group of degree $n$;
	\item $\Dih_n$ denotes the dihedral group of order $2n$;
	\item $\Cyc_n$, $n>2,$ denotes a cyclic characteristic subgroup $\Cyc_n\subset\Dih_n$ of index 2;
	\item $\OO(p,q)$ denotes the indefinite orthogonal group of all linear transformations of an $(p+q)$-dimensional real vector space that leave invariant a nondegenerate, symmetric bilinear form of signature $(p,q)$.
	\item $\PO(p,q)$ denotes the corresponding projective orthogonal group, i.e. the image of the orthogonal
	group $\OO(p,q)$ under the projectivization homomorphism $\GL_{p+q}(\RR)\to\PGL_{p+q}(\RR)$;
	\item For a scheme $X$ over $\RR$ we denote by $X_\CC$ its complexification
	\[
	\XC=X\times_{\Spec\RR}\Spec\CC
	\]
	We also make the distinction between $X_\RR$ and $X(\RR)$. The first is a scheme over $\RR$ (whose closed points are either real points or conjugate pairs of complex points of $X_\CC$), and $X(\RR)$ is the set of the real points only. For example, $\PP_\RR^n$ is projective $n$-space as a scheme over $\RR$ and $\PP^n(\RR)$ is the set of real points, typically denoted by $\RR\PP^n$.
	\item $\langle a\rangle$ denotes a cyclic group generated by $a$;
	\item $A_\bullet B$ denotes an extension of $B$ with help of a normal subgroup $A$.
\end{itemize}

\subsection*{Acknowledgments} This work was performed in Steklov Mathematical Institute and supported by the Russian Science Foundation under grant 14-50-00005. The author would like to thank Constantin Shramov for many useful discussions and remarks. The author would also like to thank Fr\'{e}d\'{e}ric Mangolte, Yuri Prokhorov, Andrey Trepalin, Alexandr Zaitsev and the anonymous referee for valuable comments.

\section{Some auxiliary results}

In this short section we collect some useful facts concerning Jordan property.  

\begin{lem}\label{lem: properties}
	The following assertions hold.	
	\begin{enumerate}
		\item If $\Gamma_1$ is a subgroup of a Jordan group $\Gamma_2$, then $\Gamma_1$ is Jordan and $\J(\Gamma_1)\leqslant\J(\Gamma_2)$.
		\item If $\Gamma_1$ is a Jordan group, and there is a surjective homomorphism $\Gamma_1\to\Gamma_2$ with finite kernel, then $\Gamma_2$ is also Jordan with $\J(\Gamma_2)\leqslant \J(\Gamma_1)$.
	\end{enumerate}
\end{lem}

The proofs are elementary and we omit them. Next let us compute some Jordan constants.

\begin{lem}\label{lem: J for simple varieties}
	One has
	\begin{enumerate}
		\item $\J(\GL_2(\kk))=60$.
		\item $\J(\GL_3(\kk))=360$.
		\item $\J(\PGL_2(\kk))=60$.
		\item $\J(\PGL_3(\kk))=360$.
	\end{enumerate}
\end{lem}
\begin{proof}
	For (1) and (2) we refer the reader to \cite{J for Gl}, where $\J(\GL_n(\kk))$ is computed for each $n$. To prove (3) and (4), we recall that $\PSL_n(\kk)\simeq\PGL_n(\kk)$ for an algebraically closed $\kk$, and apply Lemma \ref{lem: properties} (2) to the natural surjections $\SL_n(\kk)\to\PSL_n(\kk)$, $n=2,3$, and get $\J(\PGL_2(\kk))\leqslant 60$ and $\J(\PGL_3(\kk))\leqslant 360$. The required equalities are given by the simple groups $\Alt_5$ and $\Alt_6$, respectively (alternatively, one can use the well-known classification of finite subgroups of $\PGL_2(\kk)$ and $\PGL_3(\kk)$, like in \cite[Lemma 2.3.1]{PS-J3}).
\end{proof}

\section{The case of algebraically closed field}

In the remaining part of the paper we shall use the standard language of $G$-varieties (see e.g. \cite{di}). We are going to deduce our main Theorem \ref{thm: main} by regularizing the action of each finite subgroup $G\subset\Cr_2(\kk)$ on some $\kk$-rational surface $X$. Then, applying to $X$ the $G$-Minimal Model Program, we reduce to the case when $X$ is either a {\it del Pezzo surface}, or a {\it conic bundle}. 

First we focus on conic bundles. Recall, that a smooth $G$-surface $(X,G)$ admits a conic bundle structure, if there is a $G$-morphism $\pi: X\rightarrow B$, where $B$ is a smooth curve and each scheme fibre is isomorphic to a reduced conic in $\PP^2$. 

Let us also fix some notation. In this paper every automorphism of a conic bundle $\pi: X\to B$ is supposed to preserve the conic bundle structure $\pi$. We shall write $\Aut(X,\pi)$ for the corresponding automorphism group. For every finite subgroup $G\subset\Aut(X,\pi)$ there exists a short exact sequence of groups
\[
1\overset{}{\longrightarrow} G_F\overset{}{\longrightarrow} G\overset{\varphi}{\longrightarrow} G_B\overset{}{\longrightarrow} 1,
\]
where $G_B\subset\Aut(B)\cong\PGL_2(\kk)$, and $G_F$ acts by automorphisms of the generic fiber $F$. Since $G$ is finite, $G_F$ is a subgroup of $\PGL_2(\kk)$. The following result will be used in the proof of Proposition \ref{prop: CB}. We sketch the proof for the reader's convenience.
\begin{lem}[{\cite[Lemma 5.2]{serre1}}]\label{lem: serr-lemma}
	Let $g\in G$ and $h\in G_F$ be such that $g$ normalizes the cyclic group $\langle h\rangle$ generated by $h$. Then $ghg^{-1}$ is equal to $h$ or to $h^{-1}$. 
\end{lem}
\begin{proof}
	We may assume that the order $n$ of $h$ is greater than 2. The automorphism $h$ has two fixed points on $F$, which can be characterized by the eigenvalue of $h$ on their tangent spaces. Denote one of these eigenvalues by $\lambda$; the other is $\lambda^{-1}$ (these are primitive $n$-th roots of unity). The pair $\{\lambda,\lambda^{-1}\}$ is canonically associated with $h$, so the pair associated with $ghg^{-1}$ is also $\{\lambda,\lambda^{-1}\}$. But $ghg^{-1}=h^k$, hence the pair associated with $h^k$ is $\{\lambda^k,\lambda^{-k}\}$. Hence $k\equiv\pm 1\mod{n}$.
\end{proof}

\begin{prop}\label{prop: CB}
	Let $X$ be a smooth rational surface with a conic bundle structure $\pi: X\to B\cong\PP^1$. Then \[\J(\Aut(X,\pi))\leqslant 7200.\]
\end{prop}
\begin{proof}
	Let $G\subset\Aut(X,\pi)$ be a finite group. We want to find a normal abelian subgroup $A\subset G$ such that $[G:A]\leqslant 7200$. In what follows, we shall call the groups $\Alt_4$, $\Sym_4$ and $\Alt_5$ ``exceptional''. We have the following possibilities for $G_F$ and $G_B$:
	\newline
	\newline
	{\bf (i) Both $G_F$ and $G_B$ are exceptional}. Then $|G|=|G_F|\cdot |G_B|\leqslant |\Alt_5|^2=3600$ and we are done. 
	\newline
	\newline
	{\bf (ii) The group $G_B$ is exceptional, and $G_F$ is either $\ZZ/n$ or $\Dih_n$,  $n\geqslant 2$}. Then one can take $A=G_F$ (if $G_F\cong\ZZ/n$) or $A=\Cyc_n\subset\Dih_n$ (if $G_F\cong\Dih_n$), and get $[G:A]\leqslant 2\cdot 60=120$ (recall that $\Cyc_n$ is characteristic in $\Dih_n$, hence is normal in $G$).
	\newline
	\newline
	{\bf (iii) The group $G_F$ is exceptional, and $G_B$ is either $\ZZ/n$ or $\Dih_n$}, $n\geqslant 2$. Take $g\in G$ such that $\varphi(g)$ generates a cyclic subgroup $G_B'\subset G_B$ with $[G_B:G_B']\leqslant 2$. Then one has 
	the inclusions 
	\[
	\langle g\rangle\subset ({G_F})_\bullet\langle \varphi(g)\rangle\subset G.
	\]
	Since the group $\langle g\rangle$ has index at most 60 in $({G_F})_\bullet\langle \varphi(g)\rangle$, the latter group contains a characteristic abelian subgroup $A$ of index at most 3600 (see Theorem \ref{rem: Chermak-Delgado}). As the group $({G_F})_\bullet\langle \varphi(g)\rangle$ has index at most $2$ in $G$, the group $A$ is normal and of index at most 7200 in $G$.
	\newline
	\newline
	{\bf (iv) Both $G_F$ and $G_B$ are cyclic or dihedral}. Assume first that $G_F$ is not isomorphic to $\Dih_2\cong(\ZZ/2)^2$. Then $G_F$ contains a cyclic characteristic subgroup $G_F'=\langle h\rangle$ of index at most 2. Similarly, $G_B$ contains a cyclic subgroup $G_B'$ of index at most 2. Pick $g\in G$ such that $\langle\varphi(g)\rangle=G_B'$. Since $G_F'$ is normal in $G$, it is normalized by $g$. Thus, by Lemma \ref{lem: serr-lemma}, either $ghg^{-1}=h$, or $ghg^{-1}=h^{-1}$. In both cases $g^2$ commutes with $h$, so the group $\langle g^2,h\rangle$ is abelian. One has the inclusions 
	\[
	\langle g^2,h\rangle\subset\langle g,h\rangle\subset ({G_F})_\bullet\langle \varphi(g)\rangle\subset G
	\]
	The index of $\langle g^2,h\rangle$ in $(G_F)_\bullet\langle \varphi(g)\rangle$ is at most 4. By Theorem \ref{rem: Chermak-Delgado} the latter group contains an abelian characteristic subgroup $A$ of index at most 16. Since $(G_F)_\bullet\langle \varphi(g)\rangle$ has index at most 2 in $G$, the group $A$ must be normal in $G$ and of index at most 32. 
	
	Finally, let us consider the case when $G_F\cong\Dih_2\cong(\ZZ/2)^2$ and $G_B\cong\ZZ/n$ or $\Dih_n$, $n\geqslant 2$. Again, take $g\in G$ such that $\varphi(g)$ generates $G_B'$ with $[G_B:G_B']\leqslant 2$. Then one has the inclusions 
	\[
	\langle g\rangle\subset ({G_F})_\bullet\langle \varphi(g)\rangle\subset G,
	\]
	so there exist a characteristic abelian subgroup $A$ of index at most 16 in $(G_F)_\bullet\langle \varphi(g)\rangle$. Since the latter group has index at most $2$ in $G$, the group $A$ is normal and of index at most 32 in $G$.
\end{proof}

The second type of rational surfaces we shall work with are del Pezzo surfaces, which by definition are projective algebraic surfaces $X$ with ample anticanonical class $-K_X$. 

\begin{prop}\label{prop: dP}
	Let $X$ be a smooth del Pezzo surface. Then
	\[
	\J(\Aut(X))\leqslant 7200.
	\]
	If $X$ is not isomorphic to $\PP^1\times\PP^1$, then one has $\J(\Aut(X))\leqslant 360$. 
\end{prop}
\begin{proof}
	We shall consider each $d=K_X^2$ separately.
	\begin{description}
		\item[$d=9$] Then $X\cong\PP^2$ and $\J(\Aut(X))\leqslant 360$ by Lemma \ref{lem: J for simple varieties} (4).
		\item[$d=8$] If $X$ is a blow up $\pi: X\to\PP^2$ at one point, then $\pi$ is $\Aut(X)$-equivariant, so $\J(\Aut(X))\leqslant 360$ by the case $d=9$. Now let $X\cong\PP^1\times\PP^1$. Then
		\[
		\Aut(X)\cong\big (\PGL_2(\kk)\times\PGL_2(\kk)\big )\rtimes\ZZ/2.
		\]
		Let $G\subset\Aut(X)$ be a finite group and $G_0$ be its intersection with $\PGL_2(\kk)\times\PGL_2(\kk)$. We will assume that $G_0\ne 1$ (otherwise $G\cong\ZZ/2$ and everything is clear). Suppose first that $G\ne G_0$. Then $G_0$ is a normal subgroup of $G$ of index 2. Let
		\[
		\pi_i: \PGL_2(\kk)\times \PGL_2(\kk)\rightarrow\PGL_2(\kk),\ \ i=1,\ 2,
		\]
		be the projection on the $i$-th component. Then $G_0\subseteq\pi_1(G_0)\times\pi_2(G_0)$. Since $\ZZ/2$ acts on $\PGL_2(\kk)\times\PGL_2(\kk)$ by exchanging the two components, we have $\pi_1(G_0)\cong\pi_2(G_0)\cong H$, where $H$ is one of the following groups: $\ZZ/n$, $\Dih_n$, $\Alt_4$, $\Sym_4$, $\Alt_5$. In the last three cases we are done, since $|G|=2|G_0|\leqslant 2|\Alt_5|^2=7200$. If $H\cong\ZZ/n$, we take $G_0$ as a normal abelian subgroup of $G$ of index 2. Finally, if $H\cong\Dih_n$, then intersecting $G_0$ with $\Cyc_n\times\Cyc_n\subset\Dih_n\times\Dih_n$ we get an abelian subgroup $G_0'\subset G_0$ of index at most 4. Thus $G_0$ contains a characteristic abelian subgroup of index at most 16 (see e.g. Theorem \ref{rem: Chermak-Delgado}), which is normal in $G$ and has index at most 32 therein.
		
		If $G=G_0$ then we either conclude applying Proposition \ref{prop: CB} (as $G=G_0$ preserves the conic bundle structure), or use arguments similar to the previous ones (here we cannot state that $\pi_1(G)\cong\pi_2(G)$, but it is not so important). 
		\item[$d=7$] Then $X$ is a blow up $\pi: X\to\PP^2$ at two points, and $\pi$ is again $\Aut(X)$-equivariant. Therefore, $\J(\Aut(X))\leqslant 360$.
		\item[$d=6$] Then $X$ is isomorphic to the surface obtained by blowing up $\PP^2_\kk$ in three noncollinear points $p_1,p_2,p_3$. The set of $(-1)$-curves on $X$ consists of six curves which form a hexagon $\Sigma$ of lines in the anticanonical embedding $X\hookrightarrow\PP_\kk^6$. One can easily show that $\Aut(X)\cong(\kk^*)^2\rtimes\Dih_6$, where the torus $(\kk^*)^2$ comes from automorphisms of $\PP_\kk^2$ that fix all the points $p_i$, and $\Dih_6$ is the symmetry group of $\Sigma$. Therefore, $\J(\Aut(X))\leqslant|\Dih_6|=12$.
		\item[$d=5$] Then $\Aut(X)\cong\Sym_5$, so $\J(\Aut(X))=120$.
		\item[$d=4$] Then $\Aut(X)\cong (\ZZ/2)^4\rtimes\Gamma$, where $|\Gamma|\leqslant 10$, see \cite[Theorem 1.1]{hosoh-quart} or \cite[Theorem 8.6.8]{cag}. Thus $\J(\Aut(X))\leqslant 10$ and we are done.
		\item[$d=3$] Then either $|\Aut(X)|\leqslant 120$, or $\Aut(X)\cong (\ZZ/3)^3\rtimes\Sym_4$ (and $X$ is the Fermat cubic surface), see \cite[Theorem 5.3]{hosoh} or \cite[Theorem 9.5.8]{cag}. But in the latter case $\J(\Aut(X))\leqslant 24$.
		\item[$d=2$] Recall that the anticanonical map 
		$
		\psi_{|-K_X|}: X\rightarrow\PP_{\kk}^2
		$
		is a double cover branched over a smooth quartic $B\subset\PP_{\kk}^2$. It is well known that $\Aut(X)\cong\Aut(B)\times\langle\gamma\rangle$, where $\gamma$ is the Galois involution of the double cover (the {\it Geiser involution}). By Hurwitz's Theorem, $|\Aut(B)|\leqslant 84(g(B)-1)=168$, so $\J(\Aut(X))\leqslant |\Aut(X)|=336$.
		\item[$d=1$] Let $X$ be a del Pezzo surface of degree 1. Recall that the linear system $|-K_X|$ is an elliptic pencil whose base locus consists of one point, which we denote by $p$. Since $p$ is fixed by $\Aut(X)$, one has the natural faithful representation
		\[
		\Aut(X)\hookrightarrow\GL(T_pX)\cong\GL_2(\CC).
		\]
		Thus $\J(\Aut(X))\leqslant \J(\GL_2(\CC))=60$ by Lemma \ref{lem: J for simple varieties} (1).
	\end{description}
\end{proof}

\begin{sled}\label{sled: 1}
	Let $X$ be a smooth rational surface. Then $\J(\Aut(X))\leqslant 7200$.
\end{sled}
\begin{proof}
	Take a finite subgroup $G\subset\Aut(X)$. Applying to $X$ the $G$-Minimal Model Program, we may assume that $X$ is either a del Pezzo surface, or a rational surface with $G$-equivariant conic bundle structure \cite[Theorem 5]{di-perf}. Now the statement follows from Propositions \ref{prop: dP} and \ref{prop: CB}.
\end{proof}

\begin{sled}[Theorem \ref{thm: main}]\label{sled: 2}
	One has $\J(\Cr_2(\kk))=7200$.
\end{sled}
\begin{proof}
	Take a finite subgroup $G\subset\Cr_2(\kk)$. Regularizing its action (see \cite[Lemma 3.5]{di}), we may assume that $G$ acts biregularly on a smooth rational surface $X$. So, the bound $\J(\Cr_2(\kk))\leqslant 7200$ follows from Corollary \ref{sled: 1}. The equality is achieved for the group $G=(\Alt_5\times\Alt_5)\rtimes\ZZ/2$ acting on $\PP^1\times\PP^1$. Indeed, if $A$ is normal in $G$, then $A\cap (\Alt_5\times\Alt_5)$ is normal in $\Alt_5\times\Alt_5$. But every normal subgroup in the direct product of two simple non-abelian groups $H$ and $K$ is one of the groups $1_H\times 1_K,\ 1_H\times K,\ H\times 1_K,\ H\times K$. If $A$ is abelian, it must be trivial. 
\end{proof}

\section{The Jordan constants for the plane Cremona groups over $\RR$ and $\QQ$}

In recent years, growing attention has been paid to the group $\Cr_2(\RR)$. In contrast with $\Cr_2(\CC)$, there are only partial classification results for its finite subgroups at the moment (see \cite{Yas} for classification of odd order subgroups in $\Cr_2(\RR)$, and \cite{Robayo} for classification of prime order birational diffeomorphisms of $\Sph^2$). However we are still able to calculate the Jordan constant $\J(\Cr_2(\RR))$. As a by-product, we also get the Jordan constant for a closely related group $\Aut(\PP^2(\RR))$ of birational diffeomorphisms of $\PP_\RR^2$ (see Introduction), and for the group $\Cr_2(\QQ)$.

Of course, from Lemma \ref{lem: properties} (1) we immediately get
\[
\J(\Cr_2(\RR))\leqslant\J(\Cr_2(\CC))=7200.
\]
Using some elementary representation theory arguments (Lemma \ref{lem: PGL2 and PGL3}), this bound can be drastically improved. The next result is classical and we omit the proof.

\begin{lem}\label{lem: PGL2 and PGL3}
	The following assertions hold.
	\begin{enumerate}
		\item Any finite subgroup of $\GL_2(\RR)$ and $\PGL_2(\RR)$ is isomorphic either to $\ZZ/n$ or $\Dih_n$ ($n\geqslant 2$).
		\item One has $\PGL_3(\RR)\cong\SL_3(\RR)$. Any finite subgroup of $\PGL_3(\RR)$ is either cyclic, or dihedral, or one of the symmetry groups of Platonic solids $\mathfrak{A}_4$, $\mathfrak{S}_4$ or $\mathfrak{A}_5$.
	\end{enumerate}
\end{lem}

\begin{prop}\label{prop: CB Real}
	Let $X$ be a smooth real $\RR$-rational surface with a conic bundle structure $\pi: X\to \PP_\RR^1$. Then \[\J(\Aut(X,\pi))\leqslant 32,\ \ \ \ \overline{\J}(\Aut(X,\pi))\leqslant 8.\]
\end{prop}
\begin{proof}
	Let $G\subset\Aut(X,\pi)$ be a finite group. Extending scalars to $\CC$, we argue as in the proof of Proposition \ref{prop: CB}. Note that $G_F$ and $G_B$ cannot be ``exceptional'' groups by Lemma \ref{lem: PGL2 and PGL3} (1), so only the case (iv) from the proof of Proposition \ref{prop: CB} occurs. As the abelian group $\langle g^2,h\rangle$ (or $\langle g\rangle$, when $G_F\cong\Dih_2$) from this case has index at most 8 in $G$, we get the desired bound for the weak Jordan constant.
\end{proof}

We next consider real del Pezzo surfaces. For completeness sake, we also compute the weak Jordan constants for their automorphism groups. Note that for an algebraically closed field $\kk$ of characteristic 0 one has $\overline{\J}(\Aut X)\leqslant 288$ for every smooth del Pezzo surface $X$ over $\kk$, see \cite[Corollary 3.2.5]{PS-J3}.

\begin{prop}\label{prop: dP Real}
	Let $X$ be a smooth real $\RR$-rational del Pezzo surface. Then one has 
	\[
	\J(\Aut(X))\leqslant 120,\ \ \ \overline{\J}(\Aut(X))\leqslant 20.
	\]
\end{prop}	
\begin{proof}
	We again consider each $d=K_X^2$ separately. Since $\J(\Aut(X))\leqslant\J(\Aut(X_\CC))$, in most cases it will be enough to get a sharper bound for $\J(\Aut(X_\CC))$, than in the proof of Proposition \ref{prop: dP}.
	\begin{description}
		\item[$d=9$] Then $X\cong\PP^2_\RR$ and $\J(\Aut(X))=|\Alt_5|=60$ by Lemma \ref{lem: PGL2 and PGL3} (2). Clearly, $\overline{\J}(\Aut(X))=12$.
		\item[$d=8$] If $X$ is the blow up $\pi: X\to\PP^2$ at one point, then every finite subgroup $G\subset\Aut(X)$ preserves the exceptional divisor of $\pi$ isomorphic to $\PP_\RR^1$. So, we conclude by Lemma \ref{lem: PGL2 and PGL3} (1). Now assume that $X_\CC\cong\PP_\CC^1\times\PP_\CC^1$. Denote by $\Quad_{r,s}$ the smooth quadric hypersurface 
		\[
		\{[x_1:\ldots :x_{r+s}]: x_1^2+\ldots +x_r^2-x_{r+1}^2-\ldots -x_{r+s}^2=0\}\subset\PP_\RR^{r+s-1}.
		\] 
		Then $X$ is either $\Quad_{3,1}$, or $\Quad_{2,2}$. In the first case $\Aut(X)\cong\PO(3,1)$. Indeed, any automorphism of $\Quad_{3,1}$ preserves the ample anticanonical class $-K_X$, hence is induced by an automorphism of the ambient projective space. Recall that 
		\[
		\OO(3,1)=\OO(3,1)^{\uparrow}\times\langle \pm I\rangle_2,
		\]
		where $I$ is the identity matrix and $\OO(3,1)^{\uparrow}$ is the subgroup preserving the future light cone. The latter group is isomorphic to $\PO(3,1)$ and we may identify subgroups of $\PO(3,1)$ with subgroups of $\OO(3,1)$. Using classification of finite subgroups of $\OO(3,1)$ given in \cite{lorentz}, we see that every finite group $G\subset\PO(3,1)$ contains a normal abelian subgroup of index at most 60 and abelian (not necessarily normal) subgroup of index at most 12.
		
		If $X\cong \Quad_{2,2}\cong\PP_\RR^1\times\PP_\RR^1$, then
		\[
		\Aut(X)\cong\big (\PGL_2(\RR)\times\PGL_2(\RR)\big )\rtimes\ZZ/2,
		\]
		and the assertion follows from Lemma \ref{lem: PGL2 and PGL3} (1) and the same arguments as in Proposition \ref{prop: dP} (case $d=8$).
		\item[$d=7$] Then $X$ is a blow up $\pi: X\to\PP^2$ at two points. One of $(-1)$-curves on $X$ is always defined over $\RR$ and $\Aut(X)$-invariant, so we again conclude by Lemma \ref{lem: PGL2 and PGL3} (1).
		\item[$d=6$] One has $\overline{\J}(\Aut(X))\leqslant \J(\Aut(X))\leqslant\J(\Aut(X_\CC))\leqslant|\Dih_6|=12$.
		\item[$d=5$] Then $\Aut(X_\CC)\cong\Sym_5$, so $\J(\Aut(X))\leqslant 120$ and $\overline{\J}(\Aut(X))\leqslant 20$. Note that there exists a real del Pezzo surface $X$ of degree 5 with $\Aut(X)\cong\Sym_5$ (it can be obtained by blowing up $\PP_\RR^2$ at 4 real points in general position). So, both bounds are sharp.
		\item[$d=4$] As we already noticed, $\Aut(X_\CC)\cong (\ZZ/2)^4\rtimes\Gamma$, where $|\Gamma|\leqslant 10$. Thus, \[\overline{\J}(\Aut(X))\leqslant \J(\Aut(X))\leqslant\J(\Aut(X_\CC))\leqslant|\Gamma|\leqslant 10.\]
		\item[$d=3$] From this moment we prefer to give more accurate bounds for $\J(\Aut(X))$. We will need these bounds in the proof of Theorem \ref{thm: Bir-Diff}, although not all of them are needed in the present proof. One has the following possibilities for $\Aut(X_\CC)$ (see \cite[Theorem 5.3]{hosoh} or \cite[Theorem 9.5.8]{cag}):
		\begin{itemize}
			\item $|\Aut(X_\CC)|=648$, $\Aut(X_\CC)\cong(\ZZ/3)^3\rtimes\Sym_4$ and $X_\CC$ is the Fermat cubic surface
			\[
			\big\{[x_0:x_1:x_2:x_3]: x_0^3+x_1^3+x_2^3+x_3^3=0\big \}\subset\PP^3.
			\]
			Therefore, $\J(\Aut(X))\leqslant |\Sym_4|=24$. Note that $\Aut(X)\cap(\ZZ/3)^3\cong(\ZZ/3)^\ell$, where $\ell=1,2$, as $\PGL_4(\RR)$ does not contain $(\ZZ/3)^3$ (see e.g. \cite[Proposition 2.17]{Yas}). Since every representation $\Sym_4\to\GL_\ell(\FF_3)$ has non-trivial kernel, $\Aut(X)$ contains an abelian subgroup of index at most 12. We conclude that $\overline{\J}(\Aut(X))\leqslant 12$.
			\item $|\Aut(X_\CC)|=120$ and $\Aut(X_\CC)\cong\Sym_5$. Thus $\J(\Aut(X))\leqslant\J(\Sym_5)=120$ and $\overline{\J}(\Aut(X))\leqslant 20$.
			\item $|\Aut(X_\CC)|=108$ and $\Aut(X_\CC)\cong\Heis_3(3)\rtimes\ZZ/4$, where $\Heis_3(3)$ is the Heisenberg group of unipotent $3\times 3$-matrices with entries in $\mathbb{F}_3$. Being a group of order 27, the Heisenberg group $\Heis_3(3)$ has non-trivial center, which must be a normal subgroup of $\Aut(X_\CC)$. Therefore, $\J(\Aut(X))\leqslant 108/3=36.$ On the other hand, $\Heis_3(3)$ contains an abelian subgroup of order 9, so $\overline{\J}(\Aut(X))\leqslant 108/9=12$.
			\item $|\Aut(X_\CC)|=54$ and $\Aut(X_\CC)\cong\Heis_3(3)\rtimes\ZZ/2$. Similarly, one has $\overline{\J}(\Aut(X))\leqslant 54/9=6$.
			\item $|\Aut(X_\CC)|\leqslant 24$. Then every non-trivial cyclic subgroup of $\Aut(X)$ has index at most $12$.
		\end{itemize}
		\item[$d=2$] Just as in the proof of Proposition \ref{prop: dP}, one has $\Aut(X)\cong\Aut(B)\times\langle\gamma\rangle$, where $\gamma$ is the Geiser involution and $B$ is a real plane quartic curve. Since $\Aut(B)\subset\PGL_3(\RR)$ is a finite group, we can apply Lemma \ref{lem: PGL2 and PGL3} (2). Namely, if $\Aut(B)$ is ``exceptional'', take $A=\langle\gamma\rangle$ as a desired normal abelian subgroup. If $\Aut(B)\cong\ZZ/n$, take $A=\Aut(X)$. If $\Aut(B)\cong\Dih_n$, take $A=C_n\times\langle\gamma\rangle$. In all the cases $\J(\Aut(X))\leqslant 60$ and $\overline{\J}(\Aut(X))\leqslant 12$.
		
		\item[$d=1$] Since the unique base point $p$ of the linear system $|-K_X|$ must be real, we again have the natural faithful representation
		\[
		\Aut(X)\to\GL(T_pX)\cong\GL_2(\RR).
		\]
		Thus $\Aut(X)$ is either cyclic, or dihedral, and $\J(\Aut(X))\leqslant 2$.
		\end{description}
\end{proof}

\begin{sled}[Theorem \ref{thm: J for Real Cremona}]
	One has
	\[
	\J(\Cr_2(\RR))=120,\ \ \ \ \ \overline{\J}(\Cr_2(\RR))=20.
	\]
\end{sled}
\begin{proof}
	Let $G\subset\Cr_2(\RR)$ be a finite subgroup. Regularizing its action on some $\RR$-rational surface and applying the $G$-Minimal Model Program, we may assume that $G$ acts biregularly on a smooth real $\RR$-rational surface, which is either a del Pezzo surface, or a $G$-equivariant conic bundle \cite[Theorem 5]{di-perf}. From Propositions \ref{prop: CB Real} and \ref{prop: dP Real}, one gets $\J(\Cr_2(\RR))\leqslant 120$ and $\overline{\J}(\Cr_2(\RR))=20$. The equalities are given by the group $\Sym_5$, which occurs as the automorphism group of a real del Pezzo surface, obtained by blowing up $\PP_\RR^2$ at four real points in general position.
\end{proof}

\begin{sled}[Theorem \ref{thm: J for Q Cremona}]
	One has
	\[
	\J(\Cr_2(\QQ))=120,\ \ \ \ \ \overline{\J}(\Cr_2(\QQ))=20.
	\]
\end{sled}
\begin{proof}
	Clearly, $\J(\Cr_2(\QQ))\leqslant \J(\Cr_2(\RR))$, $\overline{\J}(\Cr_2(\QQ))\leqslant\overline{\J}(\Cr_2(\RR))$. Since $\Sym_5$ can be realized as the automorphism group of a degree 5 del Pezzo surface over $\QQ$, we are done.
\end{proof}

\begin{sled}[Theorem \ref{thm: Bir-Diff}]
	One has
	\begin{align*}
	\J\big(\Aut(\PP^2(\RR))\big)=60,&\ \ \ \ \ \ \overline{\J}\big(\Aut(\PP^2(\RR))\big)=12,  \\
	\J(\Quad_{3,1}(\RR))=60,&\ \ \ \ \ \ \overline{\J}(\Quad_{3,1}(\RR))=12.
	\end{align*}
\end{sled} 
\begin{proof}
	Take a finite subgroup $G\subset\Aut(\PP^2(\RR))$ and regularize its action on some smooth $\RR$-rational surface $X$. As above, we can assume that $X$ is $G$-minimal and is either a del Pezzo surface, or a surface with $G$-equivariant conic bundle structure. Moreover, since we want $X(\RR)$ to be homeomorphic to $\RR\PP^2$, we may assume by \cite[Corollary 3.4.]{kol} that $X$ is isomorphic to $\PP_\RR^2$ blown up at $k$ pairs of complex conjugate points, where $k=0,\ldots,4$ and $d=K_X^2=9-2k$. From the proof of Proposition \ref{prop: dP Real}, one easily gets that $\J(\Aut(X))\leqslant 60$ and $\overline{\J}(\Aut(X))\leqslant 12$ in all the cases, except $d=k=3,\ \Aut(X_\CC)\cong\Sym_5$, and $d=5,\ k=2,\ \Aut(X_\CC)\cong\Sym_5$. To conclude that $\J(\Aut(\PP^2(\RR)))=60$ and $\overline{\J}(\Aut(\PP^2(\RR)))=12$, it suffices to show that $\Sym_5$ cannot occur as the automorphism group of such real surfaces.
	
	If $d=5$, then $\Sym_5$ is the automorphism group of the Petersen graph of $(-1)$-curves on $X_\CC$. In our case there are only 2 real lines on $X$, so $\Aut(X)$ cannot be isomorphic to $\Sym_5$. 
	
	Assume that $d=k=3$, $\Aut(X)\cong\Sym_5$ and let $\tau\in\Aut(X)\subset\PGL_4(\RR)$ be of order 5. It is easy to see that there are exactly 3 real lines $\ell_1,\ \ell_2$ and $\ell_3$ on $X$, either forming a ``triangle'' or intersecting at a single Eckardt point. Thus $\tau$ preserves each $\ell_i$. In both cases $\tau$ stabilizes a real line, say $\ell_1$, and fixes a real point $p=\ell_1\cap\ell_2$. Restricting $\tau$ to $\ell_1$, we get an automorphism $\tau'$ of $\PP^1_\RR$ with a real fixed point. Thus either $\tau'$ has order 2, which is impossible, or $\tau$ fixes $\ell_1$ pointwise. In the latter case $\tau$ fixes $\ell_2$ pointwise too, so $\tau$ fixes pointwise the plane in $\PP_\RR^3$ spanned by $\ell_1$ and $\ell_2$. So, $\tau$ is a reflection, a contradiction.
	
	Similarly, given a finite subgroup $G\subset\Aut(\Quad_{3,1}(\RR))$, we may assume that $G$ acts biregularly either on a smooth $\RR$-rational conic bundle $X$, or on an $\RR$-rational del Pezzo surface $X$. In the former case we are done by Proposition \ref{prop: CB Real}. In the latter case, to preserve the real locus structure, the degree of $X$ should be 8, 6, 4, or 2. From the proof of Proposition \ref{prop: dP Real}, we see that $\J(\Aut(X))\leqslant 60$ and $\overline{\J}(\Aut(X))\leqslant 12$ in these cases. As usual, the equality is given by the group $\Alt_5$ acting on $\Quad_{3,1}$. 
\end{proof}

\def\bibindent{2.5em}

\end{document}